\documentclass[11pt]{article}
\usepackage{graphicx}
\usepackage{amsmath,amsthm,amssymb}
\usepackage{xypic}
\usepackage{mathabx}
\usepackage{bbm}
\usepackage{color}
\usepackage{enumerate}
\def\smallint{\begingroup\textstyle \int\endgroup}
\newtheorem*{theorem*}{Theorem}
\newtheorem*{definition*}{Definition}
\newtheorem*{proposition*}{Proposition}
\newtheorem{theorem}{Theorem}[section]
\newtheorem{proposition}[theorem]{Proposition}
\newtheorem{definition}[theorem]{Definition}
\newtheorem{example}[theorem]{Example}

\newtheorem{remark}[theorem]{Remark}
\newtheorem{corollary}[theorem]{Corollary}
\newtheorem{lemma}[theorem]{Lemma}

\DeclareRobustCommand{\intsmall}{\text{\reflectbox{$\smallint$}}}

\newcommand\blfootnote[1]{
  \begingroup
  \renewcommand\thefootnote{}\footnote{#1}
  \addtocounter{footnote}{-1}
  \endgroup
}

\begin{document}
\title{On lax limits in $\infty$-categories}
\author{John D. Berman}
\maketitle

\begin{abstract}\blfootnote{This work was supported by a National Science Foundation Postdoctoral Fellowship under grant 1803089.}
We introduce partially lax limits of $\infty$-categories, which interpolate between ordinary limits and lax limits. Most naturally occurring examples of lax limits are only partially lax; we give examples arising from enriched $\infty$-categories and $\infty$-operads. Our main result is a formula for partially lax limits and colimits in terms of the Grothendieck construction. This generalizes a formula of Lurie for ordinary limits and of Gepner-Haugseng-Nikolaus for fully lax limits.
\end{abstract}

\section{Introduction}
\noindent Many notions in ordinary category theory can be described in terms of lax limits in 2-categories: Grothendieck constructions, comma categories, Kleisli categories, and so on. In theory, this means that a great deal of category theory reduces to the study of lax limits.%Citations from ordinary category theory

Gepner-Haugseng-Nikolaus \cite{GHN} have defined lax limits in $(\infty,2)$-categories, and they have proven that the Grothendieck construction is an example:

\begin{theorem}[\cite{GHN} 1.1]\label{ThmGHN}
Let $F:I\to\text{Cat}_\infty$ be a functor with associated cocartesian fibration $\smallint F\to I$ and associated cartesian fibration $\intsmall F\to I$. Then $\smallint F$ is equivalent to the lax colimit of $F$, and $\intsmall F$ is equivalent to the oplax colimit of $F$.
\end{theorem}

\noindent In this short paper, we will generalize the notion of lax limit to encompass many \emph{other} constructions in higher category theory. If $\mathcal{C}$ is an $(\infty,2)$-category, $I$ is an $\infty$-category with some morphisms marked, and $F:I\to\mathcal{C}$ is a functor, we will define a \emph{partially lax colimit} $\text{colim}^\text{lax}(F)$, which satisfies the following universal property:
\begin{itemize}
\item there are morphisms $\alpha_i:F(i)\to\text{colim}^\text{lax}(F)$ for each $i\in I$;
\item there are 2-morphisms $\alpha_\phi:\alpha_i\Rightarrow\alpha_j F(\phi)$ for each $\phi:i\to j$ in $I$;
\item $\alpha_\phi$ is an equivalence whenever $\phi$ is marked.
\end{itemize}

\noindent In particular, any time we write $\text{colim}^\text{lax}(F)$, we are implicitly referencing a marking on the domain of $F$.

There are also some variants on this idea. If we change the direction of the 2-morphisms $\alpha_\phi$, we obtain an \emph{oplax} colimit. If we change the direction of the morphisms $\alpha_i$, we obtain lax and oplax \emph{limits}.

If \emph{all} morphisms of $I$ are marked, then all the 2-morphisms are required to be equivalences, and we recover the ordinary colimit. If \emph{no} morphisms of $I$ are marked (or if only the equivalences are marked), then none of the 2-morphisms are required to be equivalences, and we recover the fully lax colimit.

The author hopes that anyone interested in lax limits will also be interested in partially lax limits, for the following simple reason:

\begin{center} Most lax limits appearing `in nature' are only partially lax.\end{center}

\noindent In this paper, we offer the following two examples as partial justification, with the expectation that more will follow:
\begin{itemize}
\item (Enriched $\infty$-categories) If $\mathcal{V}$ is a monoidal $\infty$-category, there is an $\infty$-category $\text{Cat}^\mathcal{V}$ of $\mathcal{V}$-enriched categories. This is constructed (as by Gepner-Haugseng \cite{GH}) in two steps: First, for each set $S$, we construct $\infty$-categories $\text{Cat}^\mathcal{V}_S$ of $\mathcal{V}$-enriched categories with a fixed set $S$ of objects. These assemble into a functor $\text{Cat}^\mathcal{V}_{-}:\text{Set}^\text{op}\to\text{Cat}_\infty$. Then $$\text{Cat}^\mathcal{V}\cong\text{colim}^\text{oplax}(\text{Cat}^\mathcal{V}_{-}),$$ where $\text{Set}^\text{op}$ is marked by surjections.
\item ($\infty$-operads) Any symmetric monoidal $\infty$-category $\mathcal{C}^\otimes$ induces a functor from the category of finite pointed sets $$\mathcal{C}:\text{Fin}_\ast\to\text{Cat}_\infty,$$ where $\mathcal{C}\left\langle n\right\rangle=\mathcal{C}^{\times n}$, and the maps in $\text{Fin}_\ast$ describe the monoidal operation. If $\mathcal{O}$ is an $\infty$-operad, so that it comes with a functor to $\text{Fin}_\ast$, then there is an equivalence of $\infty$-categories $$\text{Alg}_\mathcal{O}(\mathcal{C})\cong\text{lim}^\text{lax}(\mathcal{O}\to\text{Fin}_\ast\xrightarrow{\mathcal{C}}\text{Cat}_\infty),$$ where $\mathcal{O}$ is marked by its inert morphisms.
\end{itemize}

\noindent To understand these examples, we need a way to compute lax limits in $\text{Cat}_\infty$. This is the subject of our main result:

\begin{theorem*}[Theorem \ref{ThmLaxCat}]
If $I$ is a marked $\infty$-category and $F:I\to\text{Cat}_\infty$, let $p:\smallint F\to I$ (respectively $q:\intsmall F\to I$) be the associated cocartesian (or cartesian) fibration. We say that a morphism of $\smallint F$ (or $\intsmall F$) is marked if it is $p$-cocartesian (or $q$-cartesian) and lies over a marked morphism in $I$. Then:
\begin{itemize}
\item the lax colimit is the localization of $\smallint F$ at marked morphisms;
\item the oplax colimit is the localization of $\intsmall F$ at marked morphisms;
\item the lax limit is the $\infty$-category of marked sections of $p:\smallint F\to I$;
\item the oplax limit is the $\infty$-category of marked sections of $q:\intsmall F\to I$.
\end{itemize}
\end{theorem*}

\noindent In the event that no morphisms in $I$ are marked, this theorem reduces to Theorem \ref{ThmGHN}.

Descotte, Dubuc, and Szyld \cite{DDS} have recently studied partially lax limits (which they call $\sigma$-limits) in 2-categories. They suggest that the theory will be relevant to the development of 2-topoi. As for as this author knows, their 2018 paper was the first appearance of partially lax limits in print.

We will begin with a discussion of marked $\infty$-categories in Section 2, then define lax limits in Section 3. We prove our main theorem in Section 4, and finally discuss the two main examples in Section 5.

\section{Marked $\infty$-categories}
\begin{definition}
A \emph{marked category} is a category $\mathcal{C}$ along with a specified collection of morphisms, such that all isomorphisms are marked, and any composition of marked morphisms is marked.

If $\mathcal{C},\mathcal{D}$ are marked categories, a functor $F:\mathcal{C}\rightarrow\mathcal{D}$ is \emph{marked} if it sends marked morphisms to marked morphisms.
\end{definition}

\noindent There is a 2-category $\text{Cat}^\dag$ of marked categories, marked functors, and natural isomorphisms.

\begin{definition}
A \emph{marked $\infty$-category} is an $\infty$-category $\mathcal{C}$ along with a marking on the homotopy category $h\mathcal{C}$. The $\infty$-category of marked $\infty$-categories is $\text{Cat}_\infty^\dag=\text{Cat}_\infty\times_\text{Cat}\text{Cat}^\dag$.
\end{definition}

\noindent That is, a marked $\infty$-category has some marked morphisms such that:
\begin{itemize}
\item all equivalences are marked;
\item given two equivalent morphisms $f\cong g$, $f$ is marked if and only if $g$ is;
\item marked morphisms are closed under composition;
\end{itemize}
\noindent and marked functors are functors sending marked morphisms to marked morphisms.

\begin{remark}
There are many ways to construct $\text{Cat}^\dag_\infty$. See \cite{BarwickK} 1.14 and \cite{HA} 4.1.7.1.
\end{remark}

\noindent We will be interested in multiple markings on the same $\infty$-category, so we use notation like $\mathcal{C}^\dag$ to denote a marking on $\mathcal{C}$. The basic examples are:
\begin{itemize}
\item For any $\infty$-category $\mathcal{C}$, there is a \emph{sharp} marking $\mathcal{C}^\sharp$, in which all morphisms are marked,
\item and a \emph{flat} marking $\mathcal{C}^\flat$, in which only equivalences are marked;
\item If $\mathcal{C}\to\mathcal{D}$ is a cocartesian (respectively cartesian) fibration, there is a \emph{natural} marking $\mathcal{C}^\natural$, in which the cocartesian morphisms (respectively cartesian morphisms) are marked;
\item If $\mathcal{O}$ is an $\infty$-operad (Lurie \cite{HA} writes $\mathcal{O}^\otimes$), there is an \emph{inert} marking $\mathcal{O}^\mathsection$, in which the inert morphisms are marked.
\end{itemize}

\noindent If $\mathcal{C}^\dagger,\mathcal{D}^\dagger\in\text{Cat}^\dagger_\infty$, we will write $\text{Fun}^\dagger(\mathcal{C}^\dagger,\mathcal{D}^\dagger)$ for the full subcategory of $\text{Fun}(\mathcal{C},\mathcal{D})$ spanned by marked functors.

The sharp and flat markings each promote to functors $$(-)^\sharp,(-)^\flat:\text{Cat}_\infty\to\text{Cat}^\dag_\infty.$$ There is also a forgetful functor $U:\text{Cat}_\infty^\dag\to\text{Cat}_\infty$ which forgets the marking, and a chain of adjunctions $(-)^\flat\vdash U\vdash (-)^\sharp$. Moreover, $(-)^\flat$ also has a left adjoint $|-|:\text{Cat}^\dag_\infty\to\text{Cat}_\infty$ which preserves finite products (\cite{HA} 4.1.7.2).

We regard $|\mathcal{C}^\dag|$ informally as the $\infty$-category obtained from $\mathcal{C}$ by adjoining formal inverses to all the marked morphisms.

\begin{example}
If $\mathcal{C}$ is any $\infty$-category, then $|\mathcal{C}^\flat|\cong\mathcal{C}$, and $|\mathcal{C}^\sharp|$ is the \emph{geometric realization}, or the $\infty$-groupoid built by adding inverses to all morphisms.
\end{example}

\noindent We can compute limits and colimits in $\text{Cat}^\dagger_\infty$ as so:

\begin{proposition}\label{PropMarkLim}
Let $I$ be a small $\infty$-category, $F^\dagger:I\to\text{Cat}^\dagger_\infty$ a functor, and $F:I\to\text{Cat}^\dagger_\infty\rightarrow\text{Cat}_\infty$ the composite with the forgetful functor $U$. Then $\text{colim}(F^\dagger)\cong\text{colim}(F)$ and $\text{lim}(F^\dagger)\cong\text{lim}(F)$ as underlying $\infty$-categories, and the markings may be recovered as follows:
\begin{enumerate}
\item A morphism $\phi$ of $\text{colim}(F)$ is marked if there is some $i\in I$ such that $F(i)\to\text{colim}(F)$ sends a marked morphism to one equivalent to $\phi$;
\item A morphism $\phi$ of $\text{lim}(F)$ is marked if for every $i\in I$, $\phi$ is sent by $\text{lim}(F)\to F(i)$ to a marked morphism.
\end{enumerate}
\end{proposition}

\begin{proof}
Let $\text{colim}(F)^\dagger$ be marked as in (1). By construction, the functors $F^\dagger(i)\to\text{colim}(F)^\dagger$ are marked, so they induce a marked functor $e:\text{colim}(F^\dagger)\to\text{colim}(F)^\dagger$. Since the forgetful functor $U:\text{Cat}^\dagger_\infty\to\text{Cat}_\infty$ has a right adjoint, it preserves colimits, so $e$ is an equivalence of $\infty$-categories. We need only show the following: If a morphism $\phi$ of $\text{colim}(F)$ is marked, then $\phi$ is marked in $\text{colim}(F^\dagger)$. However, if $\phi$ is marked in $\text{colim}(F)$, then by definition it arises from a marked morphism of $F^\dagger(i)$ for some $i\in I$, and since $F^\dagger(i)\to\text{colim}(F^\dagger)$ is a marked functor, therefore $\phi$ is marked in $\text{colim}(F^\dagger)$.

The proof for limits is exactly the same.
\end{proof}

\begin{corollary}
The functor $(-)^\sharp:\text{Cat}_\infty\to\text{Cat}_\infty^\dagger$ preserves small colimits. Therefore it has a right adjoint which we denote $U^\dagger:\text{Cat}_\infty^\dagger\to\text{Cat}_\infty$.
\end{corollary}

\noindent Explicitly, $U^\dagger(\mathcal{C}^\dagger)$ is the subcategory of $\mathcal{C}^\dagger$ spanned by the marked morphisms (and all objects). In conclusion, we have a chain of adjunctions $$|-|\vdash(-)^\flat\vdash U\vdash(-)^\sharp\vdash U^\dagger.$$

\section{Lax limits}
\noindent Lax limits are limits indexed by a twisted arrow category, which we will review first. Twisted arrow categories are classical, but the analogue for $\infty$-categories is due to Barwick \cite{TwAr}.

\begin{definition}
If $I$ is an $\infty$-category, let $\text{Tw}(I)\to I\times I^\text{op}$ be the right fibration associated to the functor $\text{Map}(-,-):I^\text{op}\times I\rightarrow\text{Top}$. We call $\text{Tw}(I)$ the \emph{twisted arrow $\infty$-category} of $I$.
\end{definition}

\noindent We may regard $\text{Tw}(I)$ as follows: objects are morphisms $i\xrightarrow{f} j$ in $I$, and morphisms $f\to f^\prime$ are \emph{twisted} commutative squares $$\xymatrix{
i\ar[r]^f\ar[d] &j \\
i^\prime\ar[r]_{f^\prime} &j^\prime.\ar[u]
}$$

\begin{definition}
If $I^\dag$ is marked and $i\in I$, there is an induced marking on the undercategory $I_{i/}^\dag$: a morphism is marked if the forgetful functor $I_{i/}\to I$ sends it to a marked morphism of $I$. In the same way, $I_{/i}^\dag$ is also marked.
\end{definition}

\noindent Notice that precomposition with any morphism $X\rightarrow Y$ induces a marked functor $I_{Y/}^\dag\rightarrow I_{X/}^\dag$, so that the undercategory construction (and similarly the overcategory construction) is functorial $$I_{-/}^\dag:I^\text{op}\rightarrow\text{Cat}_\infty^\dag,$$ $$I_{/-}^\dag:I\rightarrow\text{Cat}_\infty^\dag.$$ We say an $\infty$-category $\mathcal{C}$ is \emph{tensored} (respectively \emph{cotensored}) over $\text{Cat}_\infty$ if there are functors $-\otimes-:\text{Cat}_\infty\times\mathcal{C}\to\mathcal{C}$, respectively $[-,-]:\text{Cat}_\infty^\text{op}\times\mathcal{C}\to\mathcal{C}$.

\begin{definition}\label{DefLax}
Suppose $I^\dag$ is a marked small $\infty$-category, $\mathcal{C}$ is an $\infty$-category, and $F:I\rightarrow\mathcal{C}$ is a functor (of $\infty$-categories). If $\mathcal{C}$ is tensored over $\text{Cat}_\infty$, we define $$\text{colim}^\text{lax}(F)=\text{colim}\left(\text{Tw}(I)\rightarrow I^\text{op}\times I\xrightarrow{|I_{-/}^\dag|\times F}\text{Cat}_\infty\times\mathcal{C}\xrightarrow{-\otimes -}\mathcal{C}\right),$$ $$\text{colim}^\text{oplax}(F)=\text{colim}\left(\text{Tw}(I)\to I^\text{op}\times I\xrightarrow{|I_{-/}^{\text{op}\dag}|\times F}\text{Cat}_\infty\times\mathcal{C}\xrightarrow{-\otimes -}\mathcal{C}\right).$$ If $\mathcal{C}$ is cotensored over $\text{Cat}_\infty$, we define $$\text{lim}^\text{lax}(F)=\text{lim}\left(\text{Tw}(I)\rightarrow I^\text{op}\times I\xrightarrow{|I_{/-}^\dag|\times F}\text{Cat}_\infty^\text{op}\times\mathcal{C}\xrightarrow{[-,-]}\mathcal{C}\right),$$ $$\text{lim}^\text{oplax}(F)=\text{lim}\left(\text{Tw}(I)\rightarrow I^\text{op}\times I\xrightarrow{|I_{/-}^{\text{op}\dag}|\times F}\text{Cat}^\text{op}_\infty\times\mathcal{C}\xrightarrow{[-,-]}\mathcal{C}\right).$$
\end{definition}

\begin{remark}
Such a colimit (respectively limit) over the twisted arrow $\infty$-category is a \emph{coend} (respectively \emph{end}), so we may write for example: $$\text{colim}^\text{lax}(F)=\text{coend}(|I_{-/}^\dag|\otimes F(-));$$ $$\text{lim}^\text{lax}(F)=\text{end}([|I_{/-}^\dag|,F(-)]).$$ However, we won't say anything about ends and coends; see \cite{GHN} for more.
\end{remark}

\begin{example}
If $I^\flat$ has the flat marking, then $|I_{/-}^\flat|=I_{/-}$ and $|I_{-/}^\flat|=I_{-/}$, so the formulas reduce to the formulas for `fully' lax limits in \cite{GHN}.
\end{example}

\begin{proposition}\label{PropLim}
If $I^\sharp$ has the sharp marking, then $$\text{colim}^\text{lax}(F)\cong\text{colim}^\text{oplax}(F)\cong\text{colim}(F),$$ $$\text{lim}^\text{lax}(F)\cong\text{lim}^\text{oplax}(F)\cong\text{lim}(F).$$
\end{proposition}

\noindent The proposition reduces to a lemma about the twisted arrow category. We say that a functor $i:\mathcal{A}\to\mathcal{B}$ is \emph{left cofinal} if for any $F:\mathcal{B}\to\mathcal{C}$, $\text{colim}(F)\cong\text{colim}(Fi)$, and \emph{right cofinal} if the same is true for limits. See \cite{HTT} 4.1, but note that Lurie uses `cofinal' where we use `left cofinal'.

A functor $\mathcal{A}\to\mathcal{B}$ is right cofinal if and only if $\mathcal{A}^\text{op}\to\mathcal{B}^\text{op}$ is left cofinal, so that the two notions are dual.

\begin{lemma}
For any $\infty$-category $I$, projection onto the first coordinate $\pi:\text{Tw}(I)\to I$ is both left and right cofinal.
\end{lemma}

\begin{proof}
Note that $\text{Tw}(I)$ is highly asymmetric; for example, it often has an initial object, but almost never has a terminal object. Therefore, the proofs of left and right cofinality will not be similar.

Left cofinality was proven by Glasman in \cite{Glasman} 2.5; we will sketch the proof. By Quillen's Theorem A (\cite{HTT} 4.1.3.1), it suffices to prove that $\pi_{i/}=\text{Tw}(I)\times_I I_{i/}$ is weakly contractible for each $i\in I$. An object of $\pi_{i/}$ is a diagram $i\to j\to k$. Projection onto $k$ describes a functor $\pi_{i/}\to(I_{i/})^\text{op}$, and this functor has a right adjoint which sends $i\to k$ to $i\to i\to k$ in $\pi_{i/}$. Since $(I_{i/})^\text{op}$ is weakly contractible (as it has a terminal object), so is $\pi_{i/}$.

For right cofinality, it suffices to prove that $\pi_{/i}=\text{Tw}(I)\times_I I_{/i}$ is weakly contractible for each $i$. An object of $\pi_{/i}$ is a diagram $i\leftarrow j\to k$. We define functors $F_1,F_2,F_3:\pi_{/i}\to\pi_{/i}$ as follows: $$F_1(i\leftarrow j\to k)=(i\leftarrow j\to j)$$ $$F_2(i\leftarrow j\to k)=(i\leftarrow j\to i)$$ $$F_3(i\leftarrow j\to k)=(i\leftarrow i\to i),$$ where all maps $i\to i$ or $j\to j$ are the identity. There are natural transformations $\text{id}\Rightarrow F_1\Leftarrow F_2\Rightarrow F_3$ given by the vertical maps in the diagram $$
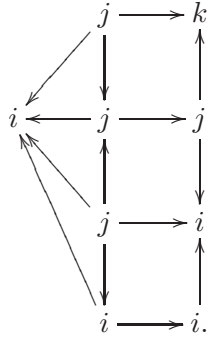
\xymatrix{
&j\ar[ld]\ar[r]\ar[d] &k \\
i &j\ar[l]\ar[r] &j\ar[u]\ar[d] \\
&j\ar[lu]\ar[r]\ar[u]\ar[d] &i \\
&i\ar[luu]\ar[r] &i.\ar[u]
}$$ This exhibits a homotopy between the identity functor and a constant functor, so $\pi_{/i}$ is weakly contractible and $\pi$ is right cofinal.
\end{proof}

\begin{proof}[Proof of Proposition \ref{PropLim}]
By definition, $|I_{i/}^\sharp|$ is the geometric realization of $I_{i/}$, which is contractible since $I_{i/}$ has an initial object. Therefore, $\text{colim}^\text{lax}(F)$ is the colimit of $\text{Tw}(I)\rightarrow I\xrightarrow{F}\mathcal{C}$, which is just the colimit of $F$ because the projection $\text{Tw}(I)\to I$ is left cofinal. The proof for oplax colimits is identical. Since $\text{Tw}(I)\to I$ is also right cofinal, the proof for lax and oplax limits is also identical.
\end{proof}

\section{Lax limits of $\infty$-categories}
\noindent The $\infty$-category $\text{Cat}_\infty$ is tensored and cotensored over itself via the functors $-\times -:\text{Cat}_\infty\times\text{Cat}_\infty\to\text{Cat}_\infty$ and $\text{Fun}(-,-):\text{Cat}_\infty^\text{op}\times\text{Cat}_\infty\to\text{Cat}_\infty$. We are nearly ready to prove our main result, that lax limits and colimits in $\text{Cat}_\infty$ can be computed explicitly via Grothendieck constructions.

\begin{definition}\label{DefIM}
Suppose that $I^\dag$ is a marked $\infty$-category and $F:I\to\text{Cat}_\infty$ is a functor. We denote by $p:\smallint F\to I$ the associated cocartesian fibration (given by the Grothendieck construction). Then the \emph{induced marking} $\smallint F^\dag$ is given as follows: A morphism $\phi$ of $\smallint F$ is marked if and only if $p(\phi)$ is marked and $\phi$ is $p$-cocartesian.

In exactly the same way there is an induced marking on the associated cartesian fibration $\intsmall F^\dag\to I^{\text{op}\dag}$.
\end{definition}

\begin{remark}
Given two functors $t:I^\dagger\to J^\dagger$ and $F:J\to\text{Cat}_\infty$, Definition \ref{DefIM} is chosen so that we have a pullback square of marked $\infty$-categories $$\xymatrix{
\smallint Ft^\dagger\ar[r]\ar[d] &\smallint F^\dagger\ar[d] \\
I^\dagger\ar[r] &J^\dagger.
}$$
\end{remark}

\begin{example}
If $I^\flat$ has the flat marking, $\smallint F^\flat$ also has the flat marking. If $I^\sharp$ has the sharp marking, $\smallint F^\natural$ has the natural marking for a cocartesian fibration (the marking by cocartesian edges). Thus, the marking $\smallint F^\dag$ interpolates between the flat marking and the natural marking.
\end{example}

\begin{theorem}\label{ThmLaxCat}
Suppose $I^\dag$ is a small marked $\infty$-category and $F:I\to\text{Cat}$ is a functor. Then $$\text{colim}^\text{lax}(F)\cong|\smallint F^\dag|,$$ $$\text{colim}^\text{oplax}(F)\cong|\intsmall F^\dag|,$$ $$\text{lim}^\text{lax}(F)\cong\text{Fun}_{/I}^\dag(I^\dag,\smallint F^\dag),$$ $$\text{lim}^\text{oplax}(F)=\text{Fun}^\dag_{/I^\text{op}}(I^{\text{op}\dag},\intsmall F^\dag).$$
\end{theorem}

\begin{remark}\label{RmkFully}
If $I^\sharp$ has the sharp marking, then the theorem describes how to compute ordinary limits and colimits in $\text{Cat}_\infty$. This is \cite{HTT} 3.3.3.2 (for limits) and \cite{HTT} 3.3.4.3 (for colimits).

If $I^\flat$ has the sharp marking, then the theorem simplifies considerably and describes how to compute fully lax limits and colimits: $\text{colim}^\text{lax}(F)\cong\smallint F$ and $\text{lim}^\text{lax}(F)\cong\text{Fun}_{/I}(I,\smallint F)$. These are the main results of \cite{GHN}.
\end{remark}

\begin{lemma}
Suppose $I$ is a small $\infty$-category and $\eta:F_0\to F_1$ is a natural transformation of functors $I\to\text{Cat}$. If each functor $\eta_i:F_0(i)\to F_1(i)$ is fully faithful, then so is $\text{lim}(\eta):\text{lim}(F_0)\to\text{lim}(F_1)$. Moreover, an object $X\in\text{lim}(F_1)$ is in the essential image of $\text{lim}(\eta)$ if and only if $\text{lim}(F_1)\to F_1(i)$ sends $X$ to an object in the essential image of $\eta_i$ for each $i\in I$.
\end{lemma}

\begin{proof}
Consider the associated diagram $$\xymatrix{
\smallint F_0^\natural\ar[rr]^{\eta_\ast}\ar[rd]_{p_0} &&\smallint F_1^\natural\ar[ld]^{p_1} \\
&I, &
}$$ where $p_i$ are cocartesian fibrations. Since $\eta_i$ is fully faithful for each $i$, so is $\eta_\ast$. By \cite{HTT} 3.3.3.2, $\text{lim}(F_i)\cong\text{Fun}_{/I}^\dagger(I^\sharp,\smallint F_i^\natural)$. Therefore, $\text{lim}(\eta)$ is equivalent to $$\text{Fun}_{/I}^\dagger(I^\sharp,\smallint F_0^\natural)\to\text{Fun}_{/I}^\dagger(I^\sharp,\smallint F_1^\natural),$$ which is fully faithful as desired.

The maps $\text{lim}(F_1)\to F_1(i)$ are given by evaluating a section at $i\in I$. Therefore, a section $s:I\to\smallint F_1$ lifts to $\smallint F_0$ if and only if $s(i)\in F_1(i)$ is in $F_0(i)\subseteq F_1(i)$ for all $i$, which proves the second part of the lemma.
\end{proof}

\begin{proof}[Proof of Theorem \ref{ThmLaxCat}]
First we prove the theorem for lax colimits. By definition, $\text{colim}^\text{lax}(F)$ is the colimit of $$F^\text{lax}_c:\text{Tw}(I)\to I^\text{op}\times I\xrightarrow{p}\text{Cat},$$ where $p(i,j)=|I_{i/}^\dagger|\times F(j)$. Since $|-|:\text{Cat}^\dagger\to\text{Cat}$ preserves finite products, $p(i,j)\cong|I_{i/}^\dagger\times F(j)^\flat|$, and therefore $F^\text{lax}_c(-)\cong|\bar{F}^\text{lax}_c(-)|$, where $$\bar{F}^\text{lax}_c:\text{Tw}(I)\to I^\text{op}\times I\xrightarrow{\bar{p}}\text{Cat}^\dag$$ and $\bar{p}(i,j)=I_{i/}^\dagger\times F(j)^\flat$. By Proposition \ref{PropMarkLim}, $\text{colim}(\bar{F}^\text{lax}_c)$ is the \emph{fully} lax colimit of $F$, which is the Grothendieck construction $\smallint F$ by Remark \ref{RmkFully}, and a morphism of $\smallint F$ is marked if and only if it is in the image of a marked morphism under $I^\dagger_{i/}\times F(j)^\flat\to\smallint F$ for some $j\to i$ in $\text{Tw}(I)$. These are the morphisms which are cocartesian over marked morphisms of $I$.

Therefore, $\text{colim}(\bar{F}^\text{lax}_c)\cong\smallint F^\dagger$. Since $|-|:\text{Cat}^\dagger\to\text{Cat}$ has a right adjoint, it preserves colimits, so $$\text{colim}^\text{lax}(F)=\text{colim}(F^\text{lax}_C)\cong|\text{colim}(\bar{F}^\text{lax}_C)|\cong|\smallint F^\dagger|.$$ The proof for oplax colimits is exactly the same. Now we turn to lax limits.

By definition, $\text{lim}^\text{lax}$ is the limit of the composite $$F^\text{lax}_\ell:\text{Tw}(I)\to I^\text{op}\times I\xrightarrow{q}\text{Cat},$$ where $q(i,j)=\text{Fun}(|I_{/i}^\dagger|,F(j))\cong\text{Fun}^\dagger(I_{/i}^\dagger,F(j)^\flat)$. In particular, consider the composite $$\bar{F}^\text{lax}_\ell:\text{Tw}(I)\to I^\text{op}\times I\xrightarrow{\bar{q}}\text{Cat}$$ where $\bar{q}(i,j)=\text{Fun}(I_{/i},F(j))$. Then there is a natural transformation $F^\text{lax}_\ell\to\bar{F}^\text{lax}_\ell$ which is a full subcategory inclusion at each $j\to i$ in $\text{Tw}(I)$. By the lemma, therefore $\text{lim}^\text{lax}(F)=\text{lim}(F^\text{lax}_\ell)$ is a full subcategory of $\text{lim}(\bar{F}^\text{lax}_\ell)$, which is the \emph{fully} lax limit. By Remark \ref{RmkFully}, the fully lax limit is the $\infty$-category of sections of $t:\smallint F\to I$, so we conclude by the lemma that $$\text{lim}^\text{lax}(F)\subseteq\text{Fun}_{/I}(I,\smallint F),$$ and a section is in $\text{lim}^\text{lax}(F)$ if and only if it sends marked morphisms to $t$-cocartesian morphisms.
\end{proof}

\section{Examples}
\noindent We will end with two examples of partially lax limits.

\subsection{Example: enriched $\infty$-categories}
\noindent Let $\mathcal{V}$ be a monoidal $\infty$-category. For each set $S$, Gepner-Haugseng \cite{GH} construct an $\infty$-operad $\mathcal{O}_S$ and define: A $\mathcal{V}$-enriched category with set $S$ of objects is an $\text{Assoc}_S$-algebra in $\mathcal{V}$. Then we may write $\text{Cat}^\mathcal{V}_S=\text{Alg}_{\mathcal{O}_S}(\mathcal{V})$ for the $\infty$-category of $\mathcal{V}$-enriched categories with set $S$ of objects. This construction is functorial in $S$, $$\text{Cat}^\mathcal{V}_{-}:\text{Set}^\text{op}\to\widehat{\text{Cat}},$$ and the $\infty$-category of \emph{all} $\mathcal{V}$-enriched categories is defined to be a localization of the associated cartesian fibration (\cite{GH} 5.4.3). By Theorem \ref{ThmLaxCat}, $\text{Cat}^\mathcal{V}$ can be described as an oplax colimit:

\begin{definition}
The $\infty$-category of $\mathcal{V}$-enriched categories is the oplax colimit $$\text{Cat}^\mathcal{V}\cong\text{colim}^\text{oplax}(\text{Cat}^\mathcal{V}_{-}),$$ where $\text{Set}^\text{op}$ is marked by the surjections.
\end{definition}

\noindent We will briefly motivate this definition. If $\mathcal{C}$ is an enriched category with set $T$ of objects, and $f:S\to T$ is a function, there is an enriched category $f^\ast\mathcal{C}$, which is determined essentially uniquely by the property: $f^\ast\mathcal{C}$ has set $S$ of objects, and there is a fully faithful functor $\alpha_\mathcal{C}:f^\ast\mathcal{C}\to\mathcal{C}$ which acts as $f:S\to T$ on the sets of objects. That is, there are triangles $$\xymatrix{
\text{Cat}^\mathcal{V}_T\ar[dd]_{f^\ast}\ar[rd] &\\
&\text{Cat}^\mathcal{V}, \\
\text{Cat}^\mathcal{V}_S\ar[ru] &
}$$ with natural transformations $\alpha$ going \emph{up} the triangle. Moreover, if $f$ is surjective, then each $\alpha_\mathcal{C}$ is fully faithful and essentially surjective, so we expect $\alpha$ to be a natural equivalence when $f$ is surjective. In other words, $\text{Cat}^\mathcal{V}$ should satisfy the same universal property as an oplax colimit with respect to the surjective marking on Set.

This will be described in greater detail in our upcoming work on enriched $\infty$-categories \cite{Berman1}.

\subsection{Example: algebras over $\infty$-operads}
\noindent Let $\text{Fin}_\ast$ denote the category of finite pointed sets. We may think of a symmetric monoidal $\infty$-category $\mathcal{C}^\otimes$ as a functor $\mathcal{C}:\text{Fin}_\ast\to\text{Cat}_\infty$ such that the maps $\mathcal{C}\left\langle n\right\rangle\to\mathcal{C}=\mathcal{C}\left\langle 1\right\rangle$ exhibit $\mathcal{C}\left\langle n\right\rangle\cong\mathcal{C}^{\times n}$. This is the same property as Segal's $\Gamma$-spaces; see  \cite{HA} 2.4.2.2.

\begin{proposition}\label{PropOp}
If $\mathcal{O}$ is an $\infty$-operad and $\mathcal{C}$ is a symmetric monoidal $\infty$-category, then the $\infty$-category of $\mathcal{O}$-algebras in $\mathcal{C}$ is the lax limit of the composite $$\text{Alg}_\mathcal{O}(\mathcal{C})\cong\text{lim}^\text{lax}(\mathcal{O}\to\text{Fin}_\ast\xrightarrow{\mathcal{C}}\text{Cat}_\infty),$$ where $\mathcal{O}$ is marked by the \emph{inert morphisms} (defined \cite{HA} 2.1.2.3).
\end{proposition}

\begin{proof}
By definition (\cite{HA} 2.1.2.7), $\text{Alg}_\mathcal{O}(\mathcal{C})\cong\text{Fun}^\dagger_{/\text{Fin}_\ast}(\mathcal{O}^\mathsection,\smallint\mathcal{C})$. Here $^\mathsection$ denotes the marking by inert morphisms, and $\smallint\mathcal{C}\to\text{Fin}_\ast$ is the cocartesian fibration associated to $\text{Fin}_\ast\xrightarrow{\mathcal{C}}\to\text{Cat}_\infty$. (Lurie writes $\mathcal{C}^\otimes$ instead of $\smallint\mathcal{C}$.) By Theorem \ref{ThmLaxCat}, this is equivalent to the lax limit described.
\end{proof}

%\noindent This description of $\text{Alg}_\mathcal{O}(\mathcal{C})$ leads to a very direct proof of:

%\begin{proposition}[\cite{HA} 3.2.3.5(2)]\label{PropPrL}
%If $\mathcal{C}$ is a presentable, closed symmetric monoidal $\infty$-category, and $\mathcal{O}$ is any $\infty$-operad, then $\text{Alg}_\mathcal{O}(\mathcal{C})$ is presentable.
%\end{proposition}

%\begin{proof}
%The proposition is a direct consequence of Proposition \ref{PropOp} along with the following lemma, which asserts that lax limits of presentable $\infty$-categories are presentable.
%\end{proof}

%\begin{lemma}
%Regard $\text{Pr}^L$ as cotensored over $\text{Cat}_\infty$ via $\text{Fun}(-,-):\text{Cat}_\infty^\text{op}\times\text{Pr}^L\to\text{Pr}^L$. (This is a functor by \cite{HTT} 5.5.3.6.) Then $\text{Pr}^L$ admits lax limits over any small marked $\infty$-category, and they are preserved by the inclusion $\text{Pr}^L\subseteq\widehat{\text{Cat}}_\infty$.
%\end{lemma}

%\begin{proof}
%Since $\text{Pr}^L$ admits small limits (\cite{HTT} 5.5.3.13), and lax limits are examples of limits, so it admits small lax limits.

%The inclusion $\text{Pr}^L\subseteq\widehat{\text{Cat}}_\infty$ preserves all small limits (also by \cite{HTT} 5.5.3.13) and is compatible with the cotensoring. Therefore, it preserves lax limits.
%\end{proof}

\end{document}